\documentclass[12pt,reqno]{amsart}

\usepackage[shortlabels]{enumitem}
\usepackage{esint}
\usepackage{physics}
\usepackage{relsize}
\usepackage[backref]{hyperref}
\usepackage{mathtools}
\usepackage{amsfonts}
\usepackage[all]{xy}
\usepackage{geometry}
\newtheorem{theorem}{Theorem}[section]
\newtheorem{lemma}[theorem]{Lemma}
\newtheorem{definition}[theorem]{Definition}

\theoremstyle{corollary}

\theoremstyle{conjecture}

\theoremstyle{problem}
\newtheorem{problem}{Problem}

\theoremstyle{assumption}

\theoremstyle{proposition}
\newtheorem{proposition}[theorem]{Proposition}
\theoremstyle{remark}
\newtheorem{remark}[theorem]{Remark}
\numberwithin{equation}{section}

\newcommand{\C}{\mathbb{C}}

\newcommand{\CP}{\mathbb{CP}}

\usepackage{xcolor}

\everymath{\displaystyle}

\begin{document}

\title{On the generalized numerical criterion}

\author{SEAN TIMOTHY PAUL}
\address{Department of Mathematics, University of Wisconsin-Madison, Madison WI, 53706}
\email{stpaul@wisc.edu}

\author{Song Sun}

\address{Department of Mathematics, University of California, Berkeley, CA 94720 } 
\email{sosun@berkeley.edu}

\author{Junsheng Zhang}
\address{Department of Mathematics, University of California, Berkeley, CA 94720 } 
\email{jszhang@berkeley.edu}

\thanks{S.S. and J. Z. are partially supported by the Simons Collaboration Grant in Special Holonomy and  NSF grant DMS-2304692.}


\maketitle
\begin{abstract}
    In this note, we give examples that demonstrate a negative answer to the generalized numerical criterion problem for pairs.
\end{abstract}
\section{Introduction}

 Let $G$ be a complex reductive algebraic group and let $V$ and $W$ be two finite-dimensional rational representations of $G$. Suppose that we have two non-zero vectors $v\in V$ and $w\in W$. We consider the group orbits 
\begin{equation*}
	\mathcal{O}_{(v, w)}:=G.[(v, w)] \subset \mathbb{P}(V \oplus W)
	\end{equation*}
	\begin{equation*}
	\mathcal{O}_v:=G.[(v, 0)] \subset \mathbb{P}(V \oplus\{0\}),	
\end{equation*}and  denote by $\overline{\mathcal O}_{(v,w)}$, $\overline{\mathcal{O}}_v$ their closures in the corresponding projective spaces. Following  \cite{Paul09} we have
\begin{definition}We say the pair $(v, w)$ is {semistable} if $\overline{\mathcal{O}}_{(v, w)} \cap \overline{\mathcal{O}}_v=\emptyset$. 
\end{definition}Observe that when $V$ is the trivial one dimensional representation $\mathbb{C}$ of $G$ and $v=1$ we have that $(1,w)$ is semistable if and only if $w$ is semistable in the usual sense of Geometric Invariant Theory (G.I.T.), namely the \emph{affine} orbit closure $\overline{G. w}\subset W$ does not contain $0$. Recall that in classical G.I.T., we have the \emph{Hilbert-Mumford numerical criterion} to detect semistability. This criterion establishes the equivalence among the following conditions: 

\begin{itemize}
    \item $0\in\overline{G. w}$.
    \item There exists a maximal algebraic torus $T$ of $G$ such that $0\in \overline{T. w}$.
    \item There exists a one parameter subgroup $\lambda$ of $G$ such that $\lim_{t\rightarrow 0}\lambda(t). w=0$.
    \item There exists a maximal algebraic torus $T$ of $G$ such that $0\notin \mathcal{N}_T(w)$.
\end{itemize} 
 Here
for a given maximal algebraic torus $T$ of $G$, and a nonzero vector $w$ in a rational $G$ representation $W$, $\mathcal{N}_T(w)\subset Lie(T)^*$ denotes the \emph{weight polytope} of $w$,

In \cite{Paul09} it was observed that the Hilbert-Mumford numerical condition for semistability admits the following generalization to pairs:
 \begin{definition}
 We say the pair $(v, w)$ is  {numerically semistable} if $\mathcal N_T(v)\subset \mathcal N_T(w)$ for all maximal tori $T$. \end{definition}
 Alternatively one can fix a maximal torus $T_0$, then $(v,w)$ is numerically semistable if and only if $\mathcal N_{T_0}(g.v)\subset \mathcal N_{T_0}(g.w)$ for all $g\in G$.
It is clear that a semistable pair $(v,w)$ is always numerically semistable. A natural question is whether the converse holds.
\begin{problem}[\cite{Paul21} Generalized numerical criterion]\label{problem1} 

\

\centering

	If $(v,w)$ is numerically semistable, does it follow that $(v,w)$ is semistable?
 
\end{problem}

The answer is known to be affirmative in  special cases, such as when $V$ is the trivial representation, or when $W=\operatorname{Sym}^2(V)$ and $w=v^{\otimes 2}$. Indeed, in these cases the question reduces to the {Hilbert-Mumford numerical criterion} above. Motivated by the work of \cite{Paul09}, it is tempting to hope that a positive answer would hold for general pairs. However, in this note we clarify that the answer is \emph{negative} in general, and therefore further conditions (at present unknown to the authors) must be imposed in order to ``numerically" characterize the semistability of pairs.

\begin{theorem}\label{t:main}We have
\begin{enumerate}
      \item [(1).] For $G=SL(2,\mathbb C)$, when $W$ is irreducible, the generalized numerical criterion holds.
    \item [(2).] For $G=SL(2,\mathbb C)$, there are examples where $W$ is reducible and the generalized numerical criterion fails.
        \item [(3).] For $G=SL(3,\mathbb C)$, there are examples where $V$ and $W$ are both irreducible and homogeneous but the generalized numerical criterion fails.
\end{enumerate}
\end{theorem}

\begin{remark}
      A positive answer to Problem \ref{problem1} has been proposed in several papers in the last decade  but unfortunately the proofs all contain gaps. To avoid possibly further confusing researchers on this topic, we list all the relevant papers that we are aware of which have made such attempts (\cite{Paul12}, \cite{Paul13}, \cite{Tian2017}, \cite{Tian2018}, \cite[Theorem 5.6]{BHJ}, \cite{BHJ-Erratum}).  We mention that in all these papers neither the irreducibility of $W$ nor the homogeneity have been used. 
\end{remark}

We now explain the consequence of this result in K\"ahler geometry. Recall that for a projective manifold $\iota: X\hookrightarrow\CP^N$, one can associate a space $\mathcal B$ of Bergman metrics on $X$. These metrics  are given by $\iota^*F^*\omega_{FS}$, where $\omega_{FS}$ is the Fubini-Study metric and $F$ is a projective linear transformation. In \cite{Paul09} it is shown that  the Mabuchi functional is uniformly bounded below on $\mathcal B$ if and only if the pair $(R_X^{\deg (\Delta_X)}, \Delta_X^{\deg (R_X)})$ is semistable. Here $R_X$ is the \emph{resultant} of $X$ and $\Delta_X$ is the \emph{hyperdiscriminant} of $X$; both can be viewed as elements in appropriate rational representations of $SL(N+1, \C)$. It is known that the corresponding representations are both irreducible and homogeneous. There was an expectation in the field, see for example the above cited papers, that the generalized numerical criterion holds for the pair $(R_X^{\deg (\Delta_X)}, \Delta_X^{\deg (R_X)})$. Such a result would establish a connection between K-semistability and CM-semistability, as well as between uniform K-stability and a uniform version of CM-stability, as discussed in \cite{BHJ}. However, the above (3) reveals that this may not be the case - at least it can not be established through a general argument and even if it can be proved, it must depend crucially on very specific knowledge of $R_X$ and $\Delta_X$. It is an interesting question whether one can find counterexamples in this particular context. We mention that in the case of Fano manifolds, we do know that K-stability implies CM-stability, but this only  follows a posteriori from the solution of the K\"ahler-Einstein problem \cite{CDS1, CDS2, CDS3}. 

We should point out that the notion of a semistable pair is equivalent to the condition of \emph{dominance} $(V\oplus W,(v,w))\succ (W,w)$ in the representation theory of reductive complex linear algebraic groups.  We refer the reader to \cite{smirnov05} for the precise definition of dominance. In \cite{smirnov05} several necessary conditions for semistability of pairs are provided. One that is particularly relevant for us asserts that if $(v,w)$ is semistable then \emph{the number of closed orbits in $\overline{\mathcal O}_{(v,w)}$ and $\overline{\mathcal{O}}_w$ coincide.} In the case where $V$ and $W$ are irreducible, homogeneous, non-isomorphic and the highest weight of $V$ is dominated by that of $W$ (in the sense of partitions \cite{Paul09}) it follows from the Borel fixed point theorem that $(v,w)$ is semistable iff there is a \textbf{unique} closed orbit in $\overline{\mathcal O}_{(v,w)}$. 

\section{The case of $SL(2, \mathbb C)$}
In this section, we suppose $G=SL(2,\mathbb C)$. Denote by $U=\C^2$ the standard representation of $G$ with standard basis $\{e_1, e_2\}$.  Denote by $T_0$ the standard torus in $G$ given by\begin{equation*}
 	T_0=\left\{\lambda(t)\equiv\begin{pmatrix}
 		t&0\\
 		0&t^{-1}
 	\end{pmatrix}: t\in \mathbb C^*\right\}.
 	 \end{equation*}
 	 For simplicity of notation, we will denote $\mathcal N(v)\equiv \mathcal N_{T_0}(v)$ the weight polytope of an element $v$ in a representation of $G$. In general  $\mathcal N(v)$ is naturally a closed interval in $\mathbb R$. 
A non-trivial irreducible representation of $G$ is of the form $\operatorname{Sym}^k\C^2$ for some positive integer $k$. It is easy to see that to study Problem \ref{problem1} without loss of generality we may always assume $V$ is irreducible so is of the form $V=\operatorname{Sym}^n\C^2$.

\subsection{Irreducible case}  
  
Here we prove Part (1) of Theorem \ref{t:main}. As we assume $W$ is irreducible, it follows that $W=\operatorname{Sym}^m\C^2$ for some positive integer $m$. 
 Suppose $(v, w)\in V\oplus W$ is numerically semistable but not semistable, we will draw a contradiction.
 We can write 
 $$v=\sum_{i=\alpha_1}^{\beta_1}c_ie_1^{i}e_2^{n-i},\quad w=\sum_{i=\alpha_2}^{\beta_2} d_ie_1^ie_2^{m-i}$$ for some $0\leq \alpha_1\leq \beta_1\leq n$,  $0\leq \alpha_2\leq\beta_2\leq m$ and $c_{\alpha_1}c_{\beta_1}d_{\alpha_2}d_{\beta_2}\neq 0$.
By the  numerical semistability of $(v,w)$ we have  $\mathcal N(v)\subset \mathcal N(w)$, which implies 
\begin{equation}\label{n-stability}
	\begin{aligned}
		&2\alpha_1-n\geq 2\alpha_2-m\\
		& n-2\beta_1\geq m-2\beta_2
	\end{aligned}
\end{equation}

Since $\overline{G.[(v, w)]}\cap \mathbb P(V\oplus \{0\})\neq \emptyset$,  one can find a holomorphic map $g: \Delta^* \rightarrow G$, where $\Delta^*$ denotes a small punctured disc in $\mathbb C$, such that $g(t).[(v,w)]$ approaches $\mathbb P(V\oplus \{0\})$ as $t\rightarrow 0$. Using the Cartan decomposition, without loss of generality we may assume 
$$g(t)=(\lambda(t))^{-p}k(t)$$ for some positive integer $p$ and $k(t)=\text{Id}+\epsilon(t)$ with $\epsilon(0)=0$. 
We can express the action of $\epsilon(t)$ as follows: 
 \begin{equation}\label{eq-action of epsilon}
     \epsilon(t).e_2=p(t)e_1+q(t)e_2,
 \end{equation}
 where $p(t)=O(t)$ and $q(t)=O(t)$. Since $p>0$, the leading order term of $g(t).v$ as $t\rightarrow 0$ arises from $g(t).(e_1^{\beta_1}e_2^{n-\beta_1})$. As $\mathcal N(v)\subset \mathcal N(w)$ and $g(t).[(v,w)]$ converges to $\mathbb P(V\oplus \{0\})$, $p(t)$ can not be identically zero. Therefore in \eqref{eq-action of epsilon} we can write  $p(t)=ct^{\tau}+o(t^{\tau})$ for $c\neq 0$.  Note that 
\begin{equation*}(\lambda(t))^{-p}.(e_1^{\beta_1} (e_2+t^\tau e_1)^{n-\beta_1})=t^{(n-2\beta_1)p}\sum_{j=0}^{n-\beta_1}t^{(\tau-2p) j}\binom{n-\beta_1}{j} e_1^{\beta_1+j}e_2^{n-\beta_1-j}.
 	\end{equation*}
 Given that $n-2\beta_1\geq m-2\beta_2$ and $g(t).[(v,w)]$ converges to $\mathbb P(V\oplus \{0\})$ as $t\rightarrow 0$, the following conditions must be satisfied:
 \begin{equation}\label{inequality}
     pn-\tau(n-\beta_1)>(2\beta_2-m)p \text{ and }  pn-\tau(n-\beta_1)>pm-\tau(m-\beta_2).
 \end{equation}
Let $\gamma_1=2\beta_1-n\leq \gamma_2=2\beta_2-m$, then \eqref{inequality} is equivalent to 
\begin{equation*}
    n-\frac{\tau}{2p}(n-\gamma_1)>\gamma_2\text{ and }n-\frac{\tau}{2p}(n-\gamma_1)>m-\frac{\tau}{2p}(m-\gamma_2).
\end{equation*}
So in particular we must have $m>n>\gamma_2$, $m-\gamma_2-n+\gamma_1>0$, and
\begin{equation}\label{eq-weight comparison}
    \frac{n-\gamma_2}{n-\gamma_1}-\frac{m-n}{m-\gamma_2-n+\gamma_1}>0.
\end{equation}
Consider the quadratic function 
$$Q(x)=(n-x)(m-x-n+\gamma_1)-(n-\gamma_1)(m-n).$$
 It has roots at $x=m$ or $x=\gamma_1$. Since $\gamma_2\in [\gamma_1, m]$,  we have $Q(\gamma_2)\leq0$. This contradicts \eqref{eq-weight comparison}.

 \begin{remark}
 As one can observe	the proof hinges on the specific situation here. The argument fails either when $W$ is reducible or when $G$ has a higher rank.
 \end{remark}

\subsection{Reducible case}
Here we give the example claimed in Part (2) of Theorem \ref{t:main}.   Fix an integer $n\geq 3$.
	Let $V=\operatorname{\operatorname{Sym}}^n\mathbb C^2$ and $W=\operatorname{\operatorname{Sym}}^n\mathbb C^2\oplus \operatorname{\operatorname{Sym}}^{n-2}\mathbb C^2$. Let $\tilde e_1$, $\tilde e_2$ denote the standard basis of the second $\mathbb C^2$ in the definition of $W$. Then we choose $$v=e_1e_2^{n-1}\in V, \hspace{0.5cm} w=e_1^n+\tilde e_2^{n-2}\in W.$$ 
$\bullet$ $(v, w)$ is numerically semistable. 
This follows from a straightforward verification. Indeed let $g=\begin{pmatrix}
	a&b\\
	c&d
\end{pmatrix}\in G$. Then we have
\begin{equation*}
	g.v=(ae_1+ce_2)(be_1+de_2)^{n-1} \quad \text{and}\quad g.w=(ae_1+ce_2)^n+(b\tilde e_1+d\tilde e_2)^{n-2}.
\end{equation*} If $ac\neq 0$, then $\mathcal N(g.w)=[-n,n]$ and $\mathcal N(g.v)\subset \mathcal N(g.w)$.  If $a=0$, then $c\neq 0$, $b\neq0$ and we have that $\mathcal N(g.v)\subset\mathcal N(g.w)=[-n,n-2]$. If $c=0$, then $a\neq 0$, $d\neq 0$ and we have $\mathcal N(g.v)\subset\mathcal N(g.w)=[-n+2,n]$.

\

\noindent $\bullet$ $(v,w)$ is not semistable. To see this we choose
\begin{equation*}
	g(t)=\lambda(t)\begin{pmatrix}
		1&0\\
		t&1
	\end{pmatrix}=\begin{pmatrix}
		t&0\\
		1&t^{-1}
	\end{pmatrix}.
\end{equation*}
Then we have
\begin{equation*}
\begin{aligned}
	g(t).v=&t^{1-n}e_2^n+t^{2-n}e_1e_2^{n-1}\\
g(t).w=&(te_1+e_2)^n+t^{2-n}\tilde e_2,
\end{aligned}
	\end{equation*}
	and one can easily show that $$\lim_{t\rightarrow 0}g(t).[v,w]\in \mathbb P(V\oplus\{0\}).$$
	
	\
	
 As we observe in this example, the crucial point is that the weight  of $g(t).v$ is different from the weight of $\lambda(t).v$, due to the small perturbation $\begin{pmatrix}
1&0\\
t&1
\end{pmatrix}$. Furthermore, this change occurs at a different rate for $v$ and $w$. It is worth noting that the small perturbation can only increase the dominant weight. This is precisely the reason why the Hilbert-Mumford numerical criterion holds for classical GIT, i.e., when there is only one representation.	

More generally we can give a heuristic reason for possible failure of the generalized numerical criterion. Namely, the numerical semistability only gives the inclusion $\mathcal N(g.v)\subset \mathcal N(g.w)$ for all $g$ in the \emph{Lie group} $G$, but for semistability we need to control the effect caused by an infinitesimal perturbation, so one would like to have the polytope inclusion $$\mathcal N(A^k.v)\subset \operatorname{conv}\left( \bigcup_{i=0}^{k} \mathcal N(A^i.w)\right)$$ for elements $A$ in the \emph{Lie algebra} of $G$ and all nonnegative integers $k$, where the right-hand side denotes the convex hull of the union. However, there is no obvious reason why this should hold.
In the above example, for $A=\begin{pmatrix}
	0&0\\1&0
\end{pmatrix}$ one can see that 
\begin{equation*}
		 \mathcal N(A.v)=\{-n\}\quad \text{and}\quad \mathcal N(A^k.w)=\{n-2k\}.
\end{equation*}
So $\mathcal N(A.v)$ is not contained in $\mathcal N(A.w)$ but is contained in $\mathcal N(A^n.w)$. This discussion motivates our construction of the examples in the next section.

\section{irreducible homogeneous examples}
In this section, we produce examples claimed in Part (3) of Theorem \ref{t:main}. We follow the notations in \cite[Chapter 12]{FH}. Denote by $U=\mathbb C^3$  the standard representation of $SL(3,\C)$. Let $\{e_1, e_2, e_3\}$ and $\{e_1^*, e_2^*, e_3^*\}$ be the standard basis of $U$ and  $U^*$. Given a matrix $A=(a_{ij})\in SL(3,\C)$, our convention on the group action is 
\begin{equation*}
    A.e_j=\sum_i a_{ij}e_i, \hspace{0.2cm}  A.e_j^*=\sum_i b_{ij}e_i, \hspace{0.2cm} 
\end{equation*}
where $B=(b_{ij})$ is given by $(A^t)^{-1}$. In the Lie algebra $\mathfrak{sl}(3, \C)$ we  denote by $E_{ij}$  the $3\times 3$ matrix whose $(i,j)$ entry is 1 and whose other entries are all 0. Then we have $E_{ij}. e_j=e_i$.

We fix $T_0$ to be the standard diagonal maximal torus of $SL(3, \C)$. Let $L_i\in Lie(T_0)^*$ be the weight vector of $e_i$.  Then $\{L_1, L_2\}$ gives a basis of $Lie(T_0)^*$ and we can identify $Lie(T_0)^*$ with $\mathbb R^2$. Notice $L_3=-L_1-L_2$. Given a representation $V$ of $SL(3, \C)$ and a nonzero $v\in V$, we denote by $\mathcal N(v)\equiv \mathcal N_{T_0}(v)$, and we denote by $\mathcal N(V)$ the weight polytope $\mathcal N(v)$ of a generic vector $v\in V$.

An irreducible representation of $SL(3, \mathbb C)$ is of the form $\Gamma_{a,b}$ for $a, b\geq 0$, where $\Gamma_{a,b}$ is the sub-representation of $\operatorname{Sym}^{a}U\otimes \operatorname{Sym}^{b} U^*=\operatorname{Sym}^aU\otimes \operatorname{Sym}^b(\Lambda^2U)$  given by the kernel of the natural contraction map $$\iota: \operatorname{Sym}^aU\otimes \operatorname{Sym}^b U^*\rightarrow \operatorname{Sym}^{a-1}U\otimes \operatorname{Sym}^{b-1}U^*.$$ 
The weight polytope $\mathcal N(\Gamma_{a,b})$ is determined by the highest weight which is $aL_1-bL_3=(a+b)L_1+bL_2$. For another irreducible representation $\Gamma_{a', b'}$, it is easy to check that they are homogeneous (in the sense of Paul \cite{Paul09}) if and only if $a+2b=a'+2b'$. Notice that for homogeneous representations, we have $\mathcal N(\Gamma_{a,b})\subset\mathcal N(\Gamma_{a',b'})$ if and only if $a\leq a'$.

We consider the two homogeneous irreducible representations of $SL(3, \mathbb C)$ $$V\equiv \Gamma_{0, k+1}, \ \ \ \ \  \ W\equiv \Gamma_{2k, 1} \ \ \ \ \ \ \  (k\geq 12).$$  Then $\mathcal N(V)$ is the triangle with vertices $-(k+1)L_j(j=1,2,3)$ and $\mathcal N(W)$ is the hexagon with vertices $2kL_i-L_j(i\neq j)$. It is easy to see that $\mathcal N(V)\subset \mathcal N(W)$ and their boundaries only touch at the vertices of $\mathcal N(V)$. 
Fix a vector $$w_0=e_2^{2k-1}(e_2\otimes e_2^*-(2k+1)e_1\otimes e_1^*+e_3\otimes e_3^*)+e_1^{2k-1}(-\frac{1}{k}e_1\otimes e_1^*+e_2\otimes e_2^*+e_3\otimes e_3^*) \in W$$ and a vector  $$v=(e_2^*)^2(e_3^*)^{k-1}\in V.$$

Notice that the non-zero weights of $w_0$ are $\eta_{1}(w_0)=(2k-1)L_2$ and $\eta_2(w_0)=(2k-1)L_1$, $v$ has weight $\eta'=-2L_2-(k-1)L_3$. Moreover $E_{23}.w_0=0$ and $E_{23}.E_{23}.v=2(e_3^*)^{k+1}$ has weight $\eta''=-(k+1)L_3$.  We define a linear functional $l_3$ on $Lie(T_0)^*$ by setting
\begin{equation*}
    l_3\left(a_1 L_1+a_2 L_2+a_3 L_3\right)=a_1+ a_2-2 a_3.
\end{equation*}
Then we have \begin{equation}\label{e:weight comparison}l_3(\eta')=2k-4<l_3(\eta_1(w_0))=l_3(\eta_2(w_0))=2k-1<l_3(\eta'')=2k+2.
\end{equation}
This is the desired property that motivates our choice of $w_0$ and $v$. Recall that we have the weight decomposition $W=\oplus_{\sigma\in Lie(T_0)^*} W_{\sigma}$.
We define 
\begin{equation}
    W'=\bigoplus_{l_3(\sigma)\leq 2k-7}W_{\sigma}.
\end{equation}

\begin{proposition}\label{main result}
   Suppose $k\geq 12$, then the pair $(v, w)$ is numerically semistable but not semistable, where $w=w_0+w'$ for $w'$ a generic vector in $W'$.
\end{proposition}

It is easy to see that $(v, w)$ is not semistable. For this we take $\lambda(t)=\operatorname{diag}(t^{-\mu},t^{-\mu},t^{2\mu})$ for $\mu\geq1$. Notice that  the composition of the action of $l$ many $E_{23}$'s on $w'\in W'$ lies in the direct sum $\bigoplus_{l_3(\sigma)\leq 2k-7+3l}W_\sigma$ for $l=1,2,3$. Then by \eqref{e:weight comparison}  as $t\rightarrow0$ we have 
$$\|\lambda(t). \exp(tE_{23}).w\|\sim |t|^{\min\{-(2k-1)\mu,3-(2k+2)\mu\}},$$
and $$
\|\lambda(t).\exp(tE_{23}).v\|\sim |t|^{2-(2k+2)\mu}.$$It follows that $(v, w)$ is not semistable.

The remaining part of this section is devoted to the proof of the fact that $(v, w)$ is numerically semistable for $w'$ generic.  We fix $k\geq 12$. For $\{\alpha,\beta,\gamma\}=\{1,2,3\}$, let $S_{\alpha \beta\gamma}$ denote the $\mathbb C$-span of $e_{\alpha}^ae_{\beta}^{2k-a}\otimes e_{\gamma}^*$ for $0\leq a\leq k-1$. We introduce 6 linear functionals:
\begin{equation*}
    \begin{aligned}
         &l_1\left(a_1 L_1+a_2 L_2+a_3 L_3\right)=a_2+ a_3-2 a_1,\ f_1\left(a_1 L_1+a_2 L_2+a_3 L_3\right)=a_2-a_3\\
         &l_2\left(a_1 L_1+a_2 L_2+a_3 L_3\right)=a_1+ a_3-2 a_2,\ f_2\left(a_1 L_1+a_2 L_2+a_3 L_3\right)=a_1-a_3\\
        &l_3\left(a_1 L_1+a_2 L_2+a_3 L_3\right)=a_1+ a_2-2 a_3,\ f_3\left(a_1 L_1+a_2 L_2+a_3 L_3\right)=a_1-a_2\\
    \end{aligned}
\end{equation*}
Then we define $S_{\alpha\beta\gamma}'$ for $\{\alpha,\beta,\gamma\}=\{1,2,3\}$ by
\begin{equation*}
\begin{aligned}
    &S_{123}'= \bigoplus_{l_3(\sigma)\geq 2k-1;f_3(\sigma)\leq -2}W_{\sigma},\  S_{213}'= \bigoplus_{l_3(\sigma)\geq 2k-1;f_3(\sigma)\geq 2}W_{\sigma}\\
    &S_{132}'= \bigoplus_{l_2(\sigma)\geq 2k-1;f_2(\sigma)\leq -2}W_{\sigma},\  S_{312}'= \bigoplus_{l_2(\sigma)> 2k-1;f_2(\sigma)\geq 2}W_{\sigma}\\
    &S_{231}'= \bigoplus_{l_1(\sigma)\geq 2k-1;f_1(\sigma)\leq -2}W_{\sigma},\  S_{312}'= \bigoplus_{l_2(\sigma)\geq 2k-1;f_1(\sigma)\geq 2}W_{\sigma}.\\
\end{aligned}
\end{equation*}

For $\{\alpha, \beta, \gamma\}=\{1, 2, 3\}$ we denote by $\Pi'_{\alpha\beta\gamma}$ the natural projection map from $W$ to $S'_{\alpha\beta\gamma}$, and given $A\in SL(3,\C)$ we denote by  $\mathcal A'_{\alpha\beta\gamma}(A)$ the affine subspace of $W'$ which consists of all elements $w'\in W'$ such that $\Pi'_{\alpha\beta\gamma}(A.(w_0+w'))=0$.

\begin{lemma}\label{step1}
	 For any $\{\alpha,\beta,\gamma\}=\{1,2,3\}$ and any fixed $A\in SL(3,\C)$, $\mathcal A'_{\alpha\beta\gamma}(A)$ has codimensin at least 9 in $W'$. 

\end{lemma}
 \begin{proof}
 Given $A\in SL(3, \C)$, for simplicity of notation we denote $\mathcal A'_{\alpha\beta\gamma}\equiv\mathcal A'_{\alpha\beta\gamma}(A)$.	Recall that we can write  $A=PUL$, where $P$ is a permutation matrix, $U$ is a unipotent upper triangular matrix and $L$ is an invertible lower triangular matrix. Note that $L$ is an isomorphism from $W'$ to $W'$ and the action of  $P$ simply permutes between different $S_{\alpha\beta\gamma}'$. Without loss of generality, we may assume $P$ is the identity matrix.

Given $U$, we let $T_{\alpha\beta\gamma}(U)$ be the linear subspace of $S_{\alpha\beta\gamma}'$ spanned by $\Pi'_{\alpha\beta\gamma}(U.w')$ for $w'\in W'$.
Let us prove the claim first for $\mathcal A_{123}'$. We write the unipotent upper triangular matrix $U$ as $(u_{ij})$. There are three cases to consider. 

\textbf{Case 1.} $u_{23}\neq 0$.
 We consider vectors $e_1^ie_3^{2k-i}\otimes e_2^*\in W'$, $i=0,\cdots,2k-2$. Note that the coefficient of $e_1^je_2^{2k-j}\otimes e_3^*$ in $\Pi'_{123}(U.(e_1^ie_3^{2k-i}\otimes e_2^*))$ is given by
\begin{equation}
	-u_{23}\sum_{i_1+i_2=j}\binom{i}{i_1}\binom{2k-i}{i_2} u_{21}^{i-i_1}u_{13}^{i_2}u_{23}^{2k-i-i_2}.
\end{equation}
Because $u_{21}=0$, this projection is always 0 if $j<i$ and always nonzero when $j=i$. It then follows that $e_1^je_2^{2k-j}\otimes e_3^*\in  T_{123}(U)$ for $j=0,\cdots,k-1$. Therefore  $\operatorname{codim} \mathcal A'_{123}\geq \dim T_{123}(U)\geq k \geq 9$.

\textbf{Case 2.} $u_{23}=0$ and $u_{13}\neq 0$. We consider vectors $e_2^ie_3^{2k-i}\otimes e_1^*\in W'$, $i=0,\cdots, 2k-2$.  The coefficient of $e_1^je_2^{2k-j}\otimes e_3^*$  in $\Pi_{123}'(U.(e_2^ie_3^{2k-i}\otimes e_1^*))$  is given by
\begin{equation}
	-u_{13}\sum_{i_1+i_2=j}\binom{i}{i_1}\binom{2k-i}{i_2} u_{12}^{i_1}u_{13}^{i_2}u_{23}^{2k-i-i_2}.
\end{equation}Because $u_{23}=0$ and $u_{13}\neq 0$, this projection is always 0 if $j<2k-i$ and always nonzero when $j=2k-i$. It follows that $\dim T_{123}(U)\geq k-2$, which again gives the conclusion.

\textbf{Case 3.} $u_{23}=u_{13}=0$. In this case it is clear that $\Pi_{123}'(U.w')=0$ for all $w'\in W'$. We claim that $\Pi'_{123}(U.L.w_0)\neq 0$, which then implies that $\dim \mathcal A'_{123}=0$.
Suppose $UL=(d_{ij})$, then by the assumptions on $U$ and $L$, we know that $d_{13}=d_{23}=0$ and $d_{22}d_{33}\neq0$. For $j=0,\cdots,k-2$, the coefficient of $e_1^je_2^{2k-1-j}e_3\otimes e_3^*$ in  $U.L.w_0$ 
is given by 
\begin{equation*}
    \binom{2k-1}{j}(d_{12}^jd_{22}^{2k-1-j}+d_{11}^jd_{21}^{2k-1-j}).
\end{equation*} Because $d_{22}$ is nonzero, it is easy to show that the above numbers can not be zero simultaneously for $j=0,\cdots,k-2$.

The proof for $\mathcal A_{213}'$ is quite similar to that of $\mathcal A_{123}'$ and we omit the details. The proofs for the remaining four spaces, namely $\mathcal A_{312}', \mathcal A_{132}', \mathcal A_{231}', \mathcal A_{321}'$ are all easier and follow a similar pattern. Indeed, in each case, we can show that $\dim T_{\alpha\beta\gamma}(U)\geq k-2$. 
 Let us just do $\mathcal A_{312}'$ for example.  We then consider vectors $e_1^ie_3^{2k-i}\otimes e_2^*\in W'$, $i=0,\cdots,2k-2$ and note that the coefficient of $e_1^je_3^{2k-j}\otimes e_2^*$ in $\Pi_{312}'.U.(e_1^ie_3^{2k-i}\otimes e_2^*)$ is given by
 \begin{equation*}
 		\sum_{i_1+i_2=j}\binom{i}{i_1}\binom{2k-i}{i_2} u_{31}^{i-i_1}u_{13}^{i_2}.
 \end{equation*} Because $u_{31}=0$, this coefficient is always 0 if $j<i$ and always nonzero when $j=i$. Thus $\dim T_{312}(U)\geq k-2$.
 \end{proof}

Recall that $v=(e_2^*)^2(e_3^*)^{k-1}$ and $\mathcal N(v)$ denotes the weight polytope of $v$. Notice that the outmost two layers in the lattice points of the  polytope of $\mathcal N(V)$ consist of vectors of the form $-kL_{\alpha}-L_j$ for $ \alpha, j\in \{1,2,3\}$. 
For $\{\alpha, \beta, \gamma\}=\{1, 2, 3\}$,  we denote by $\Pi_{\alpha\beta\gamma}$ the natural projection map from $W$ to $S_{\alpha\beta\gamma}$,  and given $A\in SL(3,\C)$ we denote by $\mathcal A_{\alpha\beta\gamma}(A)$ the affine subspace of $W'$ which consists of all elements $w'\in W'$ such that $\Pi_{\alpha\beta\gamma}(A.(w_0+w'))=0$.  For $\alpha\in \{1, 2, 3\}$ we denote
$$Y_{\alpha}=\{A\in SL(3,\C):-kL_{\alpha}-L_j\in \mathcal {N}(A.v) \ \text{for some} \ j\}.$$

\begin{lemma}\label{step2}
  For any $\alpha\in \{1,2,3\}$ and $A\in Y_\alpha$, $\mathcal A_{\beta\gamma\alpha}(A)$ has codimension at least 9 for any $\beta, \gamma$ with $\{\alpha,\beta,\gamma\}=\{1,2,3\}$.
\end{lemma}
\begin{proof} It suffices to show that the linear map $W'\rightarrow S_{\beta\gamma\alpha}$, which sends $w'$ to $\Pi_{\alpha\beta\gamma}(A.w')$  has rank at least 9 for any $A\in Y_\alpha$.  For $i=0, \cdots, 2k-2$, we denote $\xi_i\equiv e_1^ie_3^{2k-i}\otimes e_2^*\in W'$.
It is easy to see that if $A\in Y_{\alpha}$ then $b_{\alpha 2}\neq0$. Here we recall that the matrix $(b_{ij})$ is the transpose inverse of $A$.

\

\textbf{Case 1.} $b_{32}\neq 0$. We need to consider $\Pi_{123}$ and $\Pi_{213}$.
By definition, we have

$$\Pi_{123}(A.\xi_i)=\sum_{j=0}^{k-1} c_{ij}e_1^je_2^{2k-j}\otimes e_3^*, $$
where 
$$c_{ij}=b_{32}\sum_{i_1=0}^j \binom{i}{i_1}\binom{2k-i}{j-i_1}a_{11}^{i_1}a_{21}^{i-i_1}a_{13}^{j-i_1}a_{23}^{2k-i-j+i_1}.$$

Notice $b_{32}= a_{21}a_{13}-a_{11}a_{23}\neq 0$. So there are a few subcases 

 \begin{itemize}
     \item[]  \textbf{Subcase 1.} $a_{21}= 0$, then $a_{11}a_{23}\neq 0$. The only non-zero term is when $i_1=i$. We get
    $$c_{ij}=0, \text{ if } j<i \text{ and } c_{ii}=b_{32}\binom{2k-i}{j-i}a_{11}^ia_{23}^{2k-i}\neq 0.$$
   It follows that the matrix $(c_{ij})$ has rank at least $k$. 
   \item[] \textbf{Subcase 2.} $a_{23}=0$, then $a_{21}a_{13}\neq 0$. The only non-zero term is when $i_1=i+j-2k$. We get  $c_{ij}=0$ for $j<2k-i$ and $c_{i, 2k-i}\neq 0$. Again the rank of $(c_{ij})$ is at least $k-2$. 
    \item []  \textbf{Subcase 3.} $a_{21}a_{23}\neq 0$,  then we can write 
$$\tilde c_{ij}\equiv c_{ij}b_{32}^{-1} a_{21}^{-i}a_{23}^{-2k+i}=\sum_{i_1=0}^j \binom{i}{i_1}\binom{2k-i}{j-i_1}x^{i_1}y^{j-i_1},$$
where $x=a_{11}a_{21}^{-1}$ and $y=a_{13}a_{23}^{-1}$. Notice that $x\neq y$. When $j=0$ we have 
$\tilde c_{i0}=1$ for all $i$. When $j=1$ we have 
$$\tilde c_{i1}=(2k-i)y+ix. $$
For general $j$ we have 
$$\tilde c_{ij}=\sum_{l=0}^{j}p_{ij,l}(x,y)i^l$$
where $p_{ij,l}(x,y)$ is a polynomial in $x,y$, and 
$$p_{ij,j}(x,y)=\sum_{i_1=0}^j \frac{(-1)^{j-i_1}}{i_1!(j-i_1)!}x^{i_1}y^{j-i_1}=\frac{1}{j!}(x-y)^j.$$
Using the Vandemonde determinant it follows that the matrix $\det(\tilde c_{ij})_{0\leq i, j\leq k-1}\neq 0$. So the rank of the matrix $(c_{ij})$, which is of size $(2k-1)\times k$, is $k$.
 \end{itemize}
    
\

Similarly one can get the conclusion for $\Pi_{213}$.

\

\textbf{Case 2.}  $b_{22}\neq 0$. In this case we need to consider $\Pi_{132}$ and $\Pi_{312}$. 
We have 

$$\Pi_{132}(A.\xi_i)=b_{22}\sum_{j=0}^{k-1}\sum_{i_1=0}^j \binom{i}{i_1}\binom{2k-i}{j-i_1}a_{11}^{i_1}a_{31}^{i-i_1}a_{13}^{j-i_1}a_{33}^{2k-i+j-i_1}e_1^je_3^{2k-j}\otimes e_2^*$$
This is similar to \textbf{Case 1}. 

\

\textbf{Case 3.}  $b_{12}\neq 0$.  In this case we need to consider $\Pi_{231}$ and $\Pi_{321}$. 
We have

$$\Pi_{231}(A.\xi_i)=b_{12}\sum_{j=0}^{k-1}\sum_{i_1=0}^j \binom{i}{i_1}\binom{2k-i}{j-i_1}a_{21}^{i_1}a_{31}^{i-i_1}a_{23}^{j-i_1}a_{33}^{2k-i+j-i_1}e_2^{j}e_3^{2k-j}\otimes e_1^*$$
Again this is similar to \textbf{Case 1}. 
\end{proof}

We now prove  Proposition \ref{main result} using a dimension argument. Consider the algebraic subset of $SL(3,\mathbb C)\times W'$ defined by
\begin{equation}
  X_{\alpha\beta\gamma}= \{(A,w'):\Pi_{\alpha\beta\gamma}'(A.(w_0+w'))=0\}.
\end{equation}
Let $\pi_1$ denote the projection to $SL(3,\C)$ and $\pi_2$ denote the projection to $W'$. Since $X_{\alpha\beta\gamma} $ is a finite union of affine varieties (irreducible algebraic sets) and by Lemma \ref{step1} every fiber of $\pi_1$ has dimension at most $\dim W'-9$, then we know \cite[I.8.Theorem 2]{Mu}
\begin{equation}
    \dim X_{\alpha\beta\gamma}< \dim W'.
\end{equation}
In particular, $\dim \pi_2(X_{\alpha\beta\gamma})< \dim W'$. Similarly, by Lemma \ref{step2} one can show that the algebraic subset of $Y_{\alpha}\times W'$ defined by
\begin{equation*}
    Z_{\beta\gamma\alpha}=\{(A,w'): \Pi_{\beta\gamma\alpha}(A.(w_0+w'))=0, A\in Y_{\alpha}\}
\end{equation*} also satisfies 
\begin{equation}
    \dim Z_{\alpha\beta\gamma}< \dim W'.
\end{equation}Note that $Z_{\alpha\beta\gamma}$ is an algebraic set because $Y_{\alpha}$ is a complement of zeroes of polynomials and clearly $\Pi_{\beta\gamma\alpha}(A.(w_0+w'))=0$ is the intersection of zeroes of polynomials. Now for a vector $w'\in W'\setminus (\bigcup_{\alpha, \beta, \gamma} \pi_2(X_{\alpha\beta\gamma})\cup \pi_2(Z_{\alpha\beta\gamma}))$, we claim that the pair $(v,w\equiv w_0+w')$ is numerically semistable. To see this, given $A\in SL(3, \C)$ by our definition of $S'_{\alpha\beta\gamma}$ and choice of $w'$ we know that $\mathcal N(A.w)$ always contains the lattice points in $\mathcal N(V)$ except the  outmost layer. So if $A\notin Y_\alpha$ then  $\mathcal N(A.w)$ strictly contains $\mathcal N(A.v)$. On the other hand, if for some $\alpha$ it holds that $A\in Y_\alpha$, then $-k L_\alpha-L_j\in \mathcal N(A.v)$ for some $j$; in this case by our definition of $S_{\beta\gamma\alpha}$ and choice of $w'$ we see that $-(k+1)L_\alpha\in \mathcal N(A.w)$. Thus  $\mathcal N(A.w)$ always contains $\mathcal N(A.v)$.

This finishes the proof of Proposition \ref{main result}.

\

As a final remark, we mention that in the above discussion, for generic $w'\in W'$ we have indeed that the pair $(v, w_0+w')$ is numerically \emph{stable} in the sense of Paul \cite{Paul09}. We briefly recall the latter notion here. Let $M_3$ denote the direct sum of $3$-copies of the standard representation $U$ and $I=e_1\oplus e_2\oplus e_3\in M_3$. Let $V$ be a representation of $G$. We define the degree of $V$ by
\begin{equation*}
    \operatorname{deg}(V):=\min \left\{d \in \mathbb{Z}_{>0} \mid \mathcal{N}(V) \subseteq d \mathcal{N}(M_3) \right\}
\end{equation*}
\begin{definition}[\cite{Paul09}]
    A pair $(v, w)$ is (numerically) stable if and only if  there is a positive integer $m$ such that
$(I^{\otimes \operatorname{deg}(V)}\otimes v^{\otimes m} , w^{\otimes (m+1)})$ is (numerically) semistable.
\end{definition}
It is easy to see that (numerical) stability implies (numerical) semistability. In our case, we observe that $\deg(V)=2k+2$ and the weight polytope $\mathcal N(I)=\mathcal{N}(M_3)$ is the equilateral triangle with vertices given by $L_j(j=1, 2,3).$ Then to see the numerical stability of $(v, w_0+w')$, one simply notices that Lemma \ref{step1} and Lemma \ref{step2} imply that whenever $\mathcal N(A.v)$ touches $\mathcal N(A.w)$, the touching point must be some vertex (say $P$) of $\mathcal N(V)$ and there is a segment $S$ in the boundary of $(2k+2)\mathcal N(M_3)$ such that $P$ is in the interior of $S$ and $S$ is contained in $\mathcal N(A.w)$.

\bibliographystyle{plain}
\bibliography{ref.bib}

\end{document}